\documentclass{amsart}
\usepackage{graphicx}
\usepackage{amssymb,amscd,amsthm,amsxtra}
\usepackage{latexsym}
\usepackage{epsfig}
\usepackage{mathtools}
\usepackage{esint}
\usepackage{color}

\newtheorem{thm}{Theorem}[section]
\newtheorem{cor}[thm]{Corollary}
\newtheorem{lem}[thm]{Lemma}
\newtheorem{prop}[thm]{Proposition}
\theoremstyle{definition}
\newtheorem{defn}[thm]{Definition}
\theoremstyle{remark}
\newtheorem{rem}[thm]{Remark}
\numberwithin{equation}{section}
\newcommand{\be}{\begin{equation}}
\newcommand{\ee}{\end{equation}}

\newcommand{\R}{\mathbb R}

\newcommand{\ep}{\epsilon}
\newcommand{\eps}{\epsilon}

\newcommand{\p}{\partial}

\newcommand{\comment}[1]{}

\begin{document}
\title[On certain degenerate one-phase free boundary problems]{On certain degenerate one-phase free boundary problems}
\author{D. De Silva}
\address{Department of Mathematics, Barnard College, Columbia University, New York, NY 10027}
\email{\tt  desilva@math.columbia.edu}
\author{O. Savin}
\address{Department of Mathematics, Columbia University, New York, NY 10027}\email{\tt  savin@math.columbia.edu}
\begin{abstract} We develop an existence and regularity theory for a class of degenerate one-phase free boundary problems. 
In this way we unify the basic theories in free boundary problems like the classical one-phase problem, the obstacle problem, 
or more generally for minimizers of the Alt-Phillips functional. 
\end{abstract}

\maketitle
\section{Introduction}

The most basic elliptic free boundary problems arise in the study of minimizers of energy functionals 
$$J(u,\Omega)=\int_\Omega \frac 12 |\nabla u|^2 + W(u) \, dx, $$
among functions $u$ which are prescribed on the boundary $$u = \varphi \quad \mbox{ on $\p \Omega$.}$$
The potential $W(t) \ge 0$ is assumed to be nonnegative and to vanish on $(-\infty,0]$. 

If we restrict our attention to nonnegative boundary data $\varphi \ge 0$ then, the conditions on $W$ guarantee that minimizers must satisfy $u \ge 0$. The strict positivity of $u$ in the interior of $\Omega$ can be deduced from the Euler-Lagrange equation
$$\triangle u= W'(u),$$
and the strong maximum principle, whenever $W$ is of class $C^{1,1}$ at the origin. Otherwise $\{u=0\}$ can 
develop patches, and then interesting questions arise concerning the properties of the {\it free boundary} $\p \{u>0\}$.

 Historically the first such case that was analyzed systematically is the {\it obstacle problem}, that corresponds to
$$W(t)=t^+, \quad \quad \triangle u= \chi_{\{u>0\}}.$$
The optimal regularity of the solution was first obtained by Frehse in \cite{F}. The general regularity theory of the free boundary was established by Caffarelli in \cite{C} (see also \cite{C4}). 
He made use of monotonicity and convexity estimates of the solution $u$ to obtain the smoothness of the reduced part of the 
free boundary $\p^*\{u>0\}$. 

An important class of potentials which were studied later by Alt and Caffarelli are those which are discontinuous at $0$, and in the simplest form correspond to
$$ W(t)=\chi_{\{t>0\}}, \quad \quad \triangle u= 0 \quad \mbox{in $\{u>0\}$}, \quad |\nabla u|=\sqrt 2 \quad   \mbox{on $\p \{u>0\}$}.$$
This is known as the {\it one-phase} or {\it Bernoulli} free boundary problem and the smoothness of the reduced part of the free boundary was established by 
variational 
techniques in \cite{AC}. Later Caffarelli developed an alternate viscosity theory approach for the regularity of the free boundary, based on 
the Harnack inequality and regularizations by sup-convolutions \cite{C1,C2,C3}. A method based on Harnack inequality and compactness arguments was subsequently developed by the first author in \cite{D}.

Another general class of examples with free boundaries are given by the Alt-Phillips energy functional, which correspond to the power-growth potentials 
$$W(t)=(t^+)^\gamma  \quad \mbox{with $\gamma \in (0,2)$}, \quad \quad  \triangle u=\gamma u^{\gamma -1}.$$
When $\gamma \in (0,1)$ these potentials interpolate between the one-phase problem $\gamma=0$ and the obstacle problem 
$\gamma=1$. Alt and Phillips showed in \cite{AP} that a similar analysis as in the one-phase problem can be carried out in this case as well, 
and they established the smoothness of the reduced part of the free boundary.  

As observed by Alt and Phillips, after a simple change of variables 
$$w=u^{1/\beta}, \quad \quad  \beta:=\frac{2}{2-\gamma}, \quad \quad \beta \in (1,\infty),$$ 
the problem above can be viewed as a one-phase free boundary problem for $w$. It turns out that $w$ is Lipschitz and it solves a degenerate equation of the type
\begin{equation}\label{weq}
\triangle w = \frac {h(\nabla w)}{w} \quad \quad \quad \mbox{in $\{w>0\}$,}  
\end{equation}
with 
\begin{equation}\label{weq1}
\nabla w \subset \{h=0\} \quad \quad \quad \mbox{on $\p \{w>0\}$,}
\end{equation}
where $h$ is the quadratic polynomial 
$$h(p)= \frac \gamma \beta - (\beta-1) |p|^2.$$
A key feature of equation \eqref{weq} is that it remains invariant under Lipschitz scaling $\tilde w(x)=w(rx)/r$. The right hand side 
degenerates as $w$ approaches $0$ and the free boundary condition \eqref{weq1} can be understood as a natural balancing 
condition in order to seek out for Lipschitz solutions $w$.

In this paper we are interested in developing the viscosity theory for the degenerate class of one-phase free boundary 
problems \eqref{weq}-\eqref{weq1}, for general functions $h$. When $h$ is not necessarily quadratic as in the examples above, then the equation \eqref{weq} cannot be reverted back to an Alt-Phillips equation by a change of variables. Our 
main assumptions are that $h \in C^1$ and, $h \ge 0$ in a star-shaped domain $D$ and $h \le 0$ outside $D$. The free boundary 
condition $\eqref{weq1}$ then reads as $\nabla w \in \p D$. The interior regularity for solutions to \eqref{weq} is not immediate as the right hand side degenerates either as $w \to 0$ or $\nabla w \to \infty$. In our analysis we will make use of the results of Imbert and Silvestre \cite{IS} in order to establish a uniform Holder modulus of continuity for $w$. 

Equations \eqref{weq}-\eqref{weq1} do not necessarily have a variational structure, but can be thought as interpolating free 
boundary conditions for 
different exponents $\gamma$ depending on the behavior of $h$ near $\p D$, and the direction $\nu$ to the free boundary. For example, a region around $\p D$ where $h$ vanishes corresponds to the classical one-phase free boundary problem, while in a region where $h$ is quadratic corresponds to solving the Alt-Phillips free boundary problem for some exponent $\gamma$.

 We remark that the sign assumptions on the function $h$ are crucial. When $h$ changes sign across $\p D$ in the opposite 
directions, $h \le 0$ in $D$ and $h \ge 0$ outside $D$, then the problem becomes completely different and it would correspond to 
the case of negative $\gamma$'s in the Alt-Phillips functional. We will address this interesting case in a subsequent paper. 

\subsection{Set-up and definitions.} Let $D \subset \R^n$ be a bounded $C^1$ domain and let $h \in C^1(\R^n)$ vanish on 
$\Gamma:=\p D$. Assume that $0 \in D$ and 
\be\label{sign_h} h \geq 0 \quad \text{in $D$}, \quad h\leq0 \quad \text{in $\bar D^c,$}\ee 
\be\label{growth} h(p) \geq -C |p|^2, \quad C>0, \quad \text{as $|p| \to \infty$.}\ee
 Here and throughout the paper, the superscript $c$ denotes the complement of the set in $\R^n.$

We ask for $D$ to be star-shaped with respect to the origin. Precisely, given a unit direction $\nu \in \mathbb S^{n-1}$, we denote by $f(\nu) \in \R$ the positive number such that $$f(\nu) \, \nu \, \, \in \, \, \Gamma=\p D.$$ In view of the $C^1$ regularity of $D$, the function
$$f: \mathbb S^{n-1} \to \R,$$
is also $C^1$. In particular there exists a $\delta>0$ such that, 
\be\label{fnu}\delta \leq f \leq \delta^{-1},\ee
and if $x=f(\nu) \, \nu \in \Gamma$ and $\omega_x$ is the unit normal to $\Gamma$ at $x$ pointing towards $\bar D^c$, then
\be \label{star} \omega_x \cdot \nu \geq \delta.\ee
Without loss of generality we may relabel the constant $C$ in \eqref{growth}, such that the inequality holds in the whole space
\be\label{growth2} h(p) \geq -C_h |p|^2,  \quad \forall p \in \R^n.\ee

We are now ready to introduce our one-phase free boundary problem: 
find a continuous function $w \geq 0$ in $\bar B_1$ 
which is prescribed on $\p B_1$ and solves 
\be\label{general}\begin{cases}\Delta w = \dfrac{h(\nabla w)}{w} \quad \text{on $B_1^+(w):=B_1 \cap \{w>0\}$,}\\
\nabla w \in \Gamma, \quad \text{on $F(w):= \p B_1^+(w) \cap B_1$.}

\end{cases}
\ee
The two conditions above are understood in the viscosity sense and we make them precise below.
First we recall that given two continuous functions $u, \psi$ in $B_1$, we say that $\psi$ touches $u$ by below (resp. above) at $x_0 \in B_1$ if $$\psi \leq u \quad \text{(resp. $\psi \geq u)$} \quad \text{near $x_0$}, \quad \psi(x_0)=u(x_0).$$
If the first inequality is strict (except at $x_0$), we say that $\psi$ touches $u$ strictly by below (resp. above.)

The notion of viscosity solution for the interior equation is standard and in fact we will show that $w$ is locally Lipschitz and it is a classical solutions in the set $\{w>0\}$. We therefore provide only the definition of viscosity solution to the free boundary condition.

\begin{defn} We say that $w$ satisfies the free boundary condition in \eqref{general} in the viscosity sense, if given $x_0 \in F(w)$, and $\psi \in C^2$ such that $\psi^+$ touches $w$ by below (resp. by above) at $x_0$, with $|\nabla \psi(x_0)| \neq 0$, and $\nu$ denotes the unit normal to $F(\psi)$ at $x_0$ pointing towards $\{w>0\}$, then $$|\nabla \psi(x_0)| \leq f(\nu), \quad \text{i.e. $\nabla \psi(x_0) \in \bar D,$} \quad \text{supersolution property}$$ $$\text{(resp. $|\nabla \psi (x_0)| \geq f(\nu)$ i.e. $\nabla \psi(x_0) \not \in D$}, \quad \text{subsolution property}.)$$

\end{defn}

As observed earlier on, this problem is invariant under Lipschitz rescaling:
$$\tilde w (x):= \frac{w(rx)}{r},\quad x \in B_1,$$
a crucial ingredient in the body of the proofs.



\subsection{Main results.} We investigate here the question of existence and regularity of viscosity solutions to \eqref{general} together with qualitative properties of their free boundaries. The main difficulty comes from the fact that the equation is degenerate near the free boundary.

We summarize our main results below. The universal constants that appear in the theorems depend only on the dimension $n,$ the $C^1$ norm of $h$ in a neighborhood of $\Gamma$, the $C^1$ norm of $f$, the constant $C_h$ in \eqref{growth2}, and the constant $\delta$ in \eqref{fnu}-\eqref{star}. In each section we will point out the precise dependence of the constants. 
 
Existence of a non-degenerate viscosity solution (see Section 3 for the precise definition of non-degeneracy) is obtained by Perron's method. Under appropriate regularity assumptions on $D$, the free boundary of the Perron solution has finite Hausdorff dimension. Precisely, 
 
 \begin{thm}[Existence and Finite Hausdorff dimension]\label{ex_reg} Given $\phi \in C^{0,\alpha}(\p B_1)$, there exists a viscosity solution to \eqref{general} in $B_1$ with $w=\phi$ on $\p B_1$. Moreover, $w$ is non-degenerate and if the set $D$ is $C^2$ and convex then
 \be\label{haus}\mathcal H^{n-1}(F(w) \cap B_{1/2}) \leq C\ee for a $C>0$ universal (depending also on the $C^2$ norm of $f$.)
 \end{thm}
 
In fact, estimate \eqref{haus} holds for any viscosity solution which is non-degenerate, as long as $D$ is convex and $C^2$ smooth.

 Concerning the regularity of viscosity solutions we prove the following.
 
 \begin{thm}[Lipschitz regularity]\label{LL} Let $w$ be a viscosity solution to \eqref{general} in $B_1$, and assume that $F(w) \cap B_{1/8} \neq \emptyset.$ Then $w$ is locally Lipschitz in $B_{1/2}$ with universal Lipschitz norm. Moreover, $w \in C^{2,\alpha}$  in $B_{1}^+(w),$ for some $0<\alpha<1$ universal.
\end{thm}

Finally, we provide an improvement of flatness lemma which leads to the following ``flatness implies $C^{1,\alpha}$"  type result. The strategy follows the lines of \cite{D}.

\begin{thm}[Flatness implies regularity]\label{flat_thm} Let $w$ be a viscosity solution to \eqref{general} in $B_1$, with $0 \in F(w)$.  There exists $\eps_0$ universal, such that if  $w$ is $\eps$-flat, i.e.
\be\label{ft}
(f(\nu)x\cdot \nu-\eps)^+ \leq w(x) \leq (f(\nu)x \cdot \nu+\eps)^+, \quad \text{in $B_1$}, \quad \eps \leq \eps_0,\ee
then $F(w) \cap B_{1/2}$ is $C^{1,\alpha}$ graph in the $\nu$ direction with norm bounded by $C \eps$, for a universal $0 < \alpha <1$. \end{thm}

This theorem gives the regularity of the reduced boundary $\p ^*\{w>0\} \subset F(w)$ for solutions satisfying \eqref{haus} since, after a sufficiently large dilation, the flatness hypothesis \eqref{ft} is guaranteed (see Lemma \ref{reduced} in Section 4).

The paper is organized as follows. In Section 2 we provide the proof of the existence statement in Theorem \ref{ex_reg}. The following section is dedicated to Theorem \ref{LL}, while measure theoretic properties of the free boundary are determined in Section 4, completing the proof of Theorem \ref{ex_reg}.
The last three sections are devoted to Theorem \ref{flat_thm}. Precisely, in Section 5 we obtain a Harnack type inequality for ``flat" solutions of \eqref{general}. This is the key ingredient which allows us to use a linearization method to obtain in Section 6 an improvement of flatness lemma. Finally the last section is dedicated to the linear problem associated to \eqref{general}.

\section{Existence} 

In this section we use Perron's method to prove the existence of a non-degenerate viscosity solution to \eqref{general}, with a given boundary data.

Let $\phi \geq 0$ be a $C^{0,\alpha}$ function on $\p B_1$. We claim that, by choosing $\alpha$ possibly smaller, the functions
$$\psi_\phi (x):= \inf_{x_0 \in \p B_1}(\phi(x_0) + C |(x-x_0) \cdot \nu_{x_0}|^{\alpha/2}), \quad x\in \bar B_1,$$
and
$$\varphi_\phi (x):= \sup_{x_0 \in \p B_1}(\phi(x_0) - C |(x-x_0) \cdot \nu_{x_0}|^{\alpha/2})^+, \quad x \in \bar B_1,$$
are respectively a supersolution and a subsolution to \eqref{general}. Moreover, it easily follows that
\begin{equation}\label{bdata}\psi_\phi=\varphi_\phi=\phi \quad \text{on $\p B_1$,}\end{equation}
provided that we choose $C$ large, depending on the $C^{0,\alpha}$norm of $\phi.$ Here $\nu_{x_0}$ is the outer unit normal to $\p B_1$ at $x_0$.

We prove the first claim. It is readily seen that the infimum of a family of supersolutions is again a supersolution, thus it is enough to show that 
$$\Psi(x):= \phi(x_0) + C |(x-x_0) \cdot \nu_{x_0}|^{\alpha/2}, \quad x_0 \in \p B_1,$$ is a supersolution.
After a change of coordinates, let us assume that $x_0=0, \nu_{x_0} = e_n.$ Then, using the quadratic bound \eqref{growth2} of $h$, and assumption \eqref{sign_h}, we get 
$$\Delta \Psi(x)= C \frac{\alpha}{2}(\frac \alpha 2 -1) x_n^{\frac \alpha 2 - 2} \leq \frac{1}{C x_n^{\frac \alpha 2}}h(C \frac \alpha 2 x_n^{\frac \alpha 2 -1}) \leq  \frac{h(\nabla \Psi)}{\Psi},$$
as long as $C$ is large enough so that $\nabla \Psi \not \in D$, and $\alpha$ is small enough so that,
$$ \frac{\alpha}{2}(\frac \alpha 2 -1) \leq - \frac 1 4 C_h  \alpha^2.$$

The second claim follows similarly by noticing that
$$\Phi(x):=  \left(\phi(x_0) - C |(x-x_0) \cdot \nu_{x_0}|^{\alpha/2}\right)^+,$$
is a subsolution for our problem \eqref{general} since $\nabla \Phi \notin D$ and
$$\triangle \Phi >0 \ge \frac{h(\nabla \Phi)}{\Phi}.$$

\smallskip

We can now prove our existence theorem. We refer to the solution achieved in the following theorem as the Perron solution associated to $\phi.$

\begin{thm} Let $$A:=\{ \psi \in C(\bar B_1): \psi \leq \psi_\phi \ \text{is a supersolution to \eqref{general}, $\psi=\phi$ on $\p B_1$}\}$$ and set  $$w(x):= \inf_{\mathcal A}\psi(x).$$ Then $w \in C(\bar B_1)$ is a viscosity solution to \eqref{general} in $B_1$, with 
$w=\phi$ on $\p B_1.$
\end{thm}
\begin{proof}  First, since  $\psi_\phi \in \mathcal A$, $w$ is well defined. Furthermore, by the maximum principle it is easily seen that each $\psi \in \mathcal A$ satisfies $$\psi \geq \varphi_\phi.$$ Indeed, this follows from the fact that $\Phi -t$, with $\Phi$ as above ant $t >0$, cannot touch a supersolution $\psi$ by below.

Now we show that we can restrict the minimization to elements in $\mathcal A$ that are uniformly H\"older continuous of exponent $\alpha/2$. Precisely, for each $\psi \in \mathcal A$ we can construct another element of $\mathcal A$, $\bar \psi \le \psi$ which is uniformly H\"older continuous, and is given by the inf-convolution 
$$\bar \psi(x) : = \inf_{y \in \bar B_1}(\psi(y) + 2C|x-y|^{\alpha/2}) \quad x \in \bar B_1.$$ 
Clearly $\bar\psi \le \psi.$ We claim that $\bar \psi \in \mathcal A$. 

To show this, 
first notice that since $\psi_\phi \geq \psi \geq \varphi_\phi$, and \eqref{bdata} holds, then for all $y \in \bar B_1,$
\be\label{11}\psi(y) + 2C|x-y|^{\alpha/2} \geq \phi(x) \geq \psi(y) -2C |x-y|^{\alpha/2}, \quad \text{if $x \in \p B_1$}\ee
with strict inequality if $y \in B_1.$

From this we deduce that
$$\bar \psi = \phi \quad \text{on $\p B_1,$}$$
and moreover if 
\be\label{y0}
\psi(y_0)= \bar \psi (x_0) - 2C|x_0 -y_0|^{\alpha/2}, \quad x_0 \in B_1
\ee 
then
$$y_0 \in B_1.$$

Now we prove that $\bar \psi$ satisfies the equation, 
$$\Delta \bar \psi \leq \frac{h(\nabla \bar \psi)}{\bar \psi}, \quad \text{in $B^+_1(\bar \psi).$}$$ A similar argument holds for the free boundary condition.
To this aim, let $P$ be a quadratic polynomial touching $\bar \psi$ (strictly) by below at $ x_0 \in B_1^+(\bar \psi),$ and let us show that 
$$\Delta P( x_0) \leq  \frac{h(\nabla P( x_0))}{P( x_0)}.$$
Let $y_0$ be as in \eqref{y0}. We distinguish two cases. If $y_0=x_0,$ the claim is obvious since $P$ touches also $\psi$ by below at $x_0$. Otherwise, notice that 
\begin{equation}\label{grad}\nabla P(x_0) \not \in D,\end{equation} let $\eta:= x_0 - y_0$ and set,
$$P_\eta(x):= P(x+\eta)- 2C|\eta|^{\alpha/2}.$$
It is easily verified that $P_\eta$ touches $\psi$ by below at $y_0$,  hence in view of \eqref{grad},
we conclude $y_0 \in B_1 \cap \{\psi >0\}$. Thus,
$$\Delta P(x_0)=\Delta P_\eta (y_0) \leq \frac{h(\nabla P_\eta (y_0))}{P_\eta (y_0)} =\frac{h(\nabla P(x_0))}{P(x_0) -2C|\eta|^{\alpha/2}} \leq \frac{h(\nabla P(x_0))}{P(x_0)},$$ where in the last inequality we used again \eqref{grad} and \eqref{sign_h}.

The claim is proved and we conclude that the function $w$ defined above is also a H\"older continuous supersolution which coincides with $\phi$ on the boundary. It remains to show that $w$ is a subsolution. 

Let $P$ be a quadratic polynomial touching $w$ strictly by above at $x_0 \in B^+_1(w),$ and assume by contradiction that $$\Delta P(x_0) < \frac{h(\nabla P(x_0))}{P(x_0)}.$$ Then, in a small neighborhood $B_\rho$ of $x_0$, $P>0$ and 
$$\Delta P < \frac{h(\nabla P)}{P}.$$
Then for $\eps>0$ small, $$\psi := \begin{cases}
w \quad \text{in $B_1 \setminus \bar B_\rho$}\\
\min\{w, P-\eps |x-x_0|^2\} \quad \text{in $B_\rho,$}\end{cases}$$ belongs to $\mathcal A$ and it is strictly below $w$. This contradicts the minimality of $w.$

Now we check the free boundary condition. If $P^+$ touches $w$ by above at $x_0 \in F(w) \cap F(P)$, with $P\in C^2$, we need to show that $\nabla P(x_0) \not \in D.$ Suppose not, and let
$$g(x):= P(x) + C\eps d(x) - Cd^2(x)+ |x-x_0|^2 - \eps^3, \quad x\in B_\eps(x_0)$$
 with $d(x)$ the signed distance from $x$ to $F(P)$, positive in $\{P>0\}.$
 The constant $C>0$ is chosen large enough so that
 $$\Delta g \leq 0 \quad \text{in $B_\eps(x_0) \cap \{g>0\}$},$$ 
 while $\eps$ is small enough so that 
 $$\nabla g (x) \in D, \quad x \in B_\eps(x_0).$$
 Notice that 
 $$g^+ \geq 0=P^+, \quad \text{in $\{d \leq 0\} \cap B_\eps(x_0)$},$$
while 
$$g > P, \quad \text{on $\{d>0\} \cap \p B_\eps(x_0)$.}$$
Thus $g^+$ is a supersolution in $B_\eps(x_0)$ and $g^+ \geq w$ on $\p B_\eps(x_0)$. We conclude that 
$$\psi:=\begin{cases}
w \quad \text{in $B_1 \setminus \bar B_{\eps}(x_0),$}\\
\min\{g^+, w\} \quad \text{in $B_\eps(x_0)$,}\end{cases}$$ is still a supersolution which is less than $w$. Since (for $\eps$ small) $g^+ \equiv 0$ in a small neighborhood of $x_0 \in F(w)$, $\psi$ does not coincide with $w$, and this contradicts the minimality of $w$.

\end{proof}

We conclude this section with a form of non-degeneracy satisfied by our solutions. 

\begin{prop}
Let $w$ be the Perron solution  in $B_1$, such that $0 \in F(w)$. Then,
$$\max_{\p B_r} w \geq c r, \quad r< 1,$$ with $c>0$ universal.
\end{prop}
\begin{proof} By rescaling, it is enough to show that
$$\max w \geq c \quad \text{on $\p B_1$},$$ for some $c>0$ universal to be specified later.
Assume not, and let $v$ be defined as (say $n>2$)
$$v= \bar c ( (1/2)^{2-n}-|x|^{2-n}), \quad \text{for $|x| \geq 1/2$},$$ extended to zero in $B_{1/2}$,
with $\bar c$ sufficiently small universal such that $|\nabla v|(x) \in D$ for $x \in \p B_{1/2}$. We then conclude that 
$v$ is a supersolution to \eqref{general} in $B_1$. By choosing $c$ small enough we guarantee that $v \geq c > w$ on $\p B_1.$ Thus,
$$\psi:=\min\{w, v\} $$ is a supersolution to \eqref{general} with the same boundary data as $w$, and by the minimality of $w$ we find $\psi=w$. On the other hand, $w=\psi \equiv 0$ in $B_{1/2}$ and we contradict that $0 \in F(w)$.
\end{proof}


\section{Lipschitz regularity}

In this section we show that any viscosity solution to \eqref{general} is uniformly Lipschitz continuous near the free boundary. 
The strategy is to show first the H\"older continuity of the solutions in the set where $w$ is positive and then to use the free boundary condition and the scaling of the equation to obtain the Lipschitz continuity. 

The only hypotheses needed in this section are (with $D \subset B_M$ for some large $M>0,$)
\be\label{hl}
-C_h|p|^2 \le h(p) \le C \chi_{B_M}.
\ee
The starting point is that $w$ is superharmonic when $|\nabla w|$ is large, and the we can use the $L^\eps$ estimate due to Imbert and Silvestre \cite{IS} for supersolutions of uniformly elliptic equations that hold only for large gradients. 
First we recall the following Theorem 5.1 from \cite{IS}, for the special case of the Laplace operator.

\begin{thm}[Imbert-Silvestre]
\label{SI} There exist small constants $\eta_0, \xi >0$, such that if $u \in C(B_1)$ is a non-negative function satisfying (in the viscosity sense)
$$\Delta u \leq 1 \quad \text{in $B_1 \cap \{|\nabla u| \geq \eta_0\}$},$$
and $\inf_{B_{1/2}}u \leq 1,$ then 
\be\label{SI_bound} \|u\|_{L^{\xi}(B_{1/2})} \leq C.\ee  with $C, \eta_0,\xi$ depending only on the dimension $n$.
\end{thm}

With this result at hands, we can prove the following Harnack type inequality.

\begin{lem}\label{sup_inf}
Let $w>0$ solve (in the viscosity sense)
$$\Delta w = \frac{h(\nabla w)}{w}, \quad \text{in $B_1$}.$$ Then given $\sigma \geq 0$, if $w \geq \sigma$, 
\be \sup_{B_{1/2}}(w-\sigma) \leq C(1+ \inf_{B_{1/2}}(w-\sigma)),\ee with $C>0$ universal.
\end{lem}
\begin{proof}
From assumption \eqref{hl}, the function 
$$u:=\frac{w-\sigma}{M/\eta_0(1+ \inf_{B_{1/2}}(w-\sigma))}$$ 
satisfies the hypothesis of Theorem \ref{SI} (say we choose $M \geq \eta_0$), and
\be\label{ws} \|w-\sigma\|_{L^\xi(B_{1/2})} \leq C(1+\inf_{B_{1/2}}(w-\sigma)).\ee On the other hand, using the equation together with assumption \eqref{hl}, we get (in the viscosity sense) in $B_1,$
$$\Delta (w-\sigma)^\gamma = \gamma \left(\frac{w-\sigma}{w} h(\nabla w) + (\gamma -1)|\nabla w|^2 \right) $$
$$\geq \gamma \left(-C_h |\nabla w|^2 + (\gamma-1)|\nabla w|^2 \right) \geq 0,$$
as long as $\gamma>0$ is large enough. By Weak Harnack inequality for $(w-\sigma)^\gamma$ (Theorem 9.26 in \cite{GT}), 
$$\sup_{B_{1/4}} \, w-\sigma \le C(\gamma,\xi) \|w-\sigma\|_{L^\xi},$$
which combined with \eqref{ws} gives the desired bound.

\end{proof}

We are now ready to prove a H\"older continuity result for  solutions to \eqref{general}, with universal estimates. We start with an oscillation decay lemma. In what follows, given a continuous function $w$ defined in a ball $B_r$ we denote, 
$$\omega(r):= \sup_{B_r} w - \inf_{B_r} w,$$
the oscillation of $w$ on $B_r.$

\begin{lem}\label{oscholder} Let $w >0$ solve (in the viscosity sense)
$$\Delta w = \frac{h(\nabla w)}{w}, \quad \text{in $B_r$}.$$ If $\omega (r) \geq K r$, for some $K$ large universal, then
$$\omega(r/2) \leq \gamma \, \, \omega(r), \quad 0<\gamma <1.$$
\end{lem}
\begin{proof} Call
$$\tilde w(x) := \frac{w(rx)-\sigma_r}{r}, \quad x \in B_1, \quad \sigma_r:=\inf_{B_r} w,$$ then $\tilde w \geq 0$ and according to Proposition \ref{sup_inf},
\be\label{si}\sup_{B_{1/2}}\tilde w \leq C(1+ \inf_{B_{1/2}}\tilde w).\ee
Our desired claim follows if we show that for some $0<\gamma<1,$
\be\label{want}osc_{B_{1/2}} \tilde w \leq \gamma  \, \, osc_{B_1} \tilde w.\ee
Notice that $$osc_{B_1} \tilde w = \frac{\omega(r)}{r}, \quad \inf_{B_1} \tilde w=0.$$ 
If for all $x \in B_{1/2}$, we have $\tilde w (x) \geq  \frac{\omega(r)}{2Cr},$ the bound \eqref{want} trivially follows. Otherwise, 
 $\tilde w (x_0) < \frac{\omega(r)}{2Cr}$ for some $x_0 \in B_{1/2},$ and \eqref{si} yields
$$\sup_{B_{1/2}} \tilde w \leq C + \frac{\omega(r)}{2r},$$ which again implies the desired bound if we choose $K > 2C.$

\end{proof}

We can now deduce our H\"older continuity estimate.

\begin{prop}\label{holder}  Let $w >0$ solve (in the viscosity sense)
$$\Delta w = \frac{h(\nabla w)}{w}, \quad \text{in $B_1$}.$$ Then, $w \in C^{0,\alpha}(B_{1/2})$ and 
$$\|w\|_{C^{0,\alpha}(B_{1/2})} \leq C(1+w(0)),$$
with $C>0$ universal and $0<\alpha<1$ universal.
\end{prop}
\begin{proof}We wish to show that
\be\label{1}\omega(r) \leq \bar C r^\alpha, \quad r=r_k:= 2^{-k}, \quad \forall k \geq 0.\ee
We choose $\bar C \geq \max\{\omega(1), 2K\}$ with $K$ given by Lemma \ref{oscholder},
and argue by induction. 
 If, $$\omega (r) \geq K r,$$ then 
$$\omega (\frac r 2) \leq \gamma \omega(r) \leq \gamma \bar C r^\alpha \leq \bar C(\frac r 2)^\alpha,$$
so by choosing $\alpha=\alpha(\gamma)$ appropriately, our claim is satisfied.
If $\omega(r)< Kr$, the claim is obviously satisfied (by our choice of $\bar C$).

The final estimate follows as in view of Lemma \ref{sup_inf}, $$\omega(\frac 12) \leq C (1+ w(0)).$$

\end{proof}

Next we deduce that solutions to our free boundary problem grow at most linearly away from the free boundary.

\begin{prop}\label{lg}
Let $w$ be a viscosity solution to \eqref{general} in $B_1$, then 
$$w(x) \leq C d(x), \quad d(x):=dist(x, F(w)), \quad B_{d(x)}(x) \subset B_{3/4},$$ with $C>0$ universal.
\end{prop}
\begin{proof}
By a Lipschitz rescaling we can assume that $0 \in B_{4/3}^+(w)$, and that $B_1$ is the largest ball around $0$ contained in $B_{4/3}^+(w),$ tangent to $F(w)$ say at $x_0$. Thus, we need to show that $w(0)$ is bounded above by a universal constant. We claim that if $w(0) \gg C$, with $C$ the constant in Lemma \ref{sup_inf} then it follows from that lemma (applied with $\sigma=0$) that $$w \geq c \, \, w(0) \quad \text{on $\p B_{1/2}$},$$
for some $c$ universal.
Now set,
$$\psi(x):= M(|x|^{-n} -1) \quad \text{in $|x| \geq 1/2$},$$ and we have 
$$\Delta \psi >0, \quad |\nabla \psi| \geq M,\quad \text{in $A:= B_1 \setminus \bar B_{1/2}$},$$ 
hence (see \eqref{hl}),
$$\Delta \psi > 0 \ge \frac{h(\nabla \psi)}{\psi} \quad \text{in $A$}.$$
If we assume by contradiction that $w(0)$ is sufficiently large, then $\psi-t$ with $t>0$ cannot touch $w$ by below in $A$ or on $\p B_{1/2}$, and it follows that $$\psi \leq w \quad \text{on $A$.}$$ Thus $\psi^+$ touches $w$ by below at $x_0 \in F(w) \cap F(\psi^+)$ and $|\nabla \psi| \in \R^n \setminus \bar D.$ This contradicts the free boundary condition for $w.$

\end{proof}

Finally, we show that solutions are locally Lipschitz with universal bound, and $C^{1,\alpha}$ in their positive phase.

\begin{thm}\label{lipth} Let $w$ be a viscosity solution to \eqref{general} in $B_1$ with $h$ satisfying \eqref{hl}. Assume that $F(w) \cap B_{1/4} \neq \emptyset.$ Then $w$ is locally Lipschitz in $B_{1/2}$ with universal Lipschitz norm. Moreover, $w \in C^{1,\alpha}$  in $B_{1}^+(w),$ for some $0<\alpha<1$. 
\end{thm}

\begin{rem}
Since $w \in C^{1,\alpha}_{loc}$ in the set $\{w>0\}$, we can then apply Schauder estimates and conclude that if $h \in C^\alpha_{loc}$ then $w \in C_{loc}^{2,\alpha}$ is a classical solution in $B^+(w)$.
\end{rem}

\begin{proof} We prove the estimates near a point $x_0 \in B_{1/2}^+(w)$. In view of Proposition \ref{lg}, the Lipschitz rescaling 
$$\tilde w(x):=\lambda^{-1}w(x_0 + \lambda x), \quad \quad \lambda:=w(x_0)$$ satisfies
$$\tilde w (0) = 1, \quad \tilde w > 0 \quad \text{in $B_{r}$}, \quad \text{ for some $r$ universal.}$$ Thus, by Proposition \ref{holder}, we find $\tilde w \in C^{0,\alpha}$ with 
$$\|\tilde w\|_{C^{0,\alpha}(B_{r/2})} \leq C$$and in particular 
$$\frac 1 2 \leq \tilde w \leq 2, \quad \text{in $B_{\rho}$, for all $\rho$ sufficiently small.}$$
Now set
$$\bar w(x)= \frac{\tilde w(\rho x)- 1}{C \rho^\alpha}, \quad x \in B_1,$$
and then
$$|\Delta \bar w| \leq 2 C^{-1}\rho^{2-\alpha} |h (C \rho^{\alpha-1}\nabla \bar w(x))|,\quad  \quad |\bar w | \leq 1 \quad \text{in $B_1$.}$$
Thus, in view of \eqref{growth2}, we can choose $\rho$ small universal, such that 
$$|\Delta \bar w| \leq \eta, \quad \text{when $|\nabla \bar w| \leq \frac 1 \eta.$}$$
with $\eta(n)>0$ the universal constant in Lemma \ref{delta1} below.
Hence, $\bar w$ is $C^{1,\alpha}$ and the desired conclusion easily follows.

\end{proof}

The following lemma says that if $u$ is almost harmonic except possibly in the region where the gradients are large, then it is of class $C^{1,\alpha}$. Its proof follows from the 
perturbations arguments developed in \cite{S}. Here we only sketch the main ideas. 

\begin{lem}\label{delta1} There exists $\eta>0$ universal (i.e. depending only on $n$), such that if $u$ solves in the viscosity sense in $B_1$,
\be\label{delta1}|\Delta u| \leq \eta, \quad \text{when $|\nabla u| \leq \frac 1 \eta,$}\ee
and $\|u\|_{L^
\infty}\leq 1$, then $u \in C^{1,\alpha}(B_{1/2}),$ for some $0<\alpha<1$ universal.
\end{lem}
\begin{proof} It suffices to show that there exists a linear function $l$ with $|\nabla l|\leq C$ universal, such that for $r>0$ small universal, and $0<\alpha<1$ universal,
\be\label{l} |u-l| \leq r^{1+\alpha} \quad \quad \text{in $B_r$}.\ee

Then it is enough to observe that 
$$\tilde u(x) := \frac{(u-l)(rx)}{r^{1+\alpha}}, \quad x \in B_1,$$ satisfies the assumptions of the lemma hence estimate \eqref{l} can be iterated indefinitely, leading to the $C^{1,\alpha}$ estimate.

Let $r>0$ be fixed, to be made precise later. Assume by contradiction that there exists a sequence $\eta_j \to 0$ as $j \to \infty$ and a sequence $u_j$ of solutions to \eqref{delta1}, with $|u_j| \leq 1$, which do not satisfy the conclusion \eqref{l}.

\medskip

We divide the proof in two steps.

\medskip

\textit{Step 1. Improvement of oscillation.} 

\smallskip

{\it Claim 1}: There exist universal constants $C_0, C_1>0$ such that if  $u \ge 0$ and 
$$u(x_0) \leq 1, \quad x_0 \in B_{1/2},$$ and 
$$|\Delta u| \leq 1, \quad \text{when $|\nabla u| \leq C_1$}$$ then 
$$|\{u< C_0\} \cap B_{1/2}| \geq \frac 3 4 |B_{1/2}|.$$

This follows from a version of the Alexandrov-Bakelman-Pucci estimate, see \cite{S}. A precise reference is Theorem 2.4 in \cite{DS}, by noticing that $u$ is a ``supersolution" in the sense of Definition 2.1 of \cite{DS}, with $\Lambda=4n$, $r$ arbitrarily small, and $I=[1,+\infty)$: 
 $u$ cannot be touched by below in a $B_r$ neighborhood by a polynomial of the form $aP$ with $a \in I$, and 
$$P:= \frac{\Lambda}{2}(x\cdot \xi)^2 - \frac 1 2 |x|^2 + L(x)$$
with $\xi$ a unit direction and $L(x):=b \cdot x + d, |b|, |d| \leq 1.$ 

\medskip

After dividing $u$ by $4C_0$ we can restate the claim as follows.

\smallskip

{\it Claim 2}: There exist universal constants $c>0$ (small) and $C_2>0$, such that if 
$$u(x_0) \leq c, \quad x_0 \in B_{1/2},$$ and 
$$|\Delta u| \leq c, \quad \text{when $|\nabla u| \leq C_2$}$$ then 
$$|\{u< \frac 1 4\} \cap B_{1/2}| \geq \frac 3 4 |B_{1/2}|.$$

Thus if $|u| \leq 1$, and if there exist $x_0, x_1 \in B_{1/2}$ such that $$(1-u)(x_0) \leq c, \quad (1+u)(x_0) \leq c$$ we would reach a contradiction as in view of Claim 2, 
$$|\{u>\frac 3 4\} \cap B_{1/2}|, |\{u<-\frac 3 4\} \cap B_{1/2}| \geq \frac 3 4 |B_{1/2}|.$$
In conclusion either $$u < 1-c \quad \text{or} \quad u>-1+c \quad \text{in $B_{1/2}$},$$
that is $$osc_{B_{1/2}} u \leq 2-c.$$

\

{\it Step 2. Compactness.} 
Consider the rescalings
$$\tilde u_{j,k}(x) := \frac{1}{(2-c)^k} u_j(2^{-k}x), \quad k\geq 0, \quad x\in B_1,$$
and apply Claim 2 inductively on $k$. We have $$|\Delta \tilde u_{j,k}|(x) = \frac{2^{-2k}}{(2-c)^k}|\Delta u_j|(2^{-k}x) \leq \eta_j \leq c$$
and $ |\nabla \tilde u_{j,k}|(x) \leq C_2 $ as long as $k$ satisfies 
$$C_2 \le 2^{-k}(2-c)^{k}\frac{1}{\eta_j}.$$ Thus, by Ascoli-Arzela and Claim 2, we conclude that $u_j$ converges (up to a subsequence) uniformly on compacts to a H\"older continuous function $u_\infty$ which satisfies
$$\Delta u_\infty = 0 \quad \text{in $B_{1/2}$.}$$
By elliptic regularity, 
$$|u_\infty - l| \leq C r^2 \quad \text{in $B_r$,} \quad \text{$l$ linear and $|\nabla l| \leq C$ universal.}$$
From the uniform convergence, 
$$|u_j -l| \leq 2C r^2 \le r^{1+\alpha}  \quad \text{in $B_r$},$$
provided that $r, \alpha$ are chosen appropriately, and we reached a contradiction.
\end{proof}

\section{Measure theoretic properties of the free boundary}

In this section we show that if $D$ is convex and $C^2$ smooth, and $h \in C^1$ then the free boundary of a non-degenerate viscosity solution to \eqref{general} has finite Hausdorff 
measure. We follow a strategy inspired by the work of Alt and Phillips in \cite{AP}. The universal constants in this section depend on the $C^2$ norm of $D$ and the $C^1$ norm of $h$ 
in a neighborhood of $\p D$.  

We say that a solution $w$ to \eqref{general} is non-degenerate if there exists a constant $\kappa>0$ such that for any $x_0 \in F(w)$ and $r$ such that $B_r(x_0) \subset B_1$ we have
\be\label{nondeg}
 x_0 \in F(w) \quad \Rightarrow \quad \max_{\p B_r(x_0)} w \geq \kappa r.
\ee

\begin{thm}\label{T7}
Assume that $D$ is a bounded convex set with $C^2$ boundary, and let $w$ be a viscosity solution to \eqref{general} satisfying the non-degeneracy condition \eqref{nondeg}.
Then
$$ \mathcal H^{n-1}(\p \{w>0\} \cap B_{1/2}) \le C(\kappa),$$
for some $C(\kappa)>0$ depending on the universal constants and $\kappa$.
\end{thm}

For this we first prove the following lemma.

\begin{lem}\label{l71}
Assume that $w$ is a global Lipschitz solution to \eqref{general}. Then $\nabla w \in \overline D$. 
\end{lem}

\begin{proof} Since $D$ is convex, it suffices to show that $\nabla w$ belongs to the convex hull of $\overline D.$
Let $L:=\sup w_n$ and assume by contradiction that $$L > \max_{y \in D} e_n \cdot y,$$ which means that 
\begin{equation}\label{70}
\mbox{$h(y) \le 0$ if $y_n \ge L$.}
\end{equation} 
Let $x_k$ be a sequence of points for which $w_n$ approaches the limit $L$. For each $k$ we rescale $w$ into
$$w^k(x):=\frac {1}{r_k} w(x_k + r_k x), \quad r_k=w(x_k),$$
so that $$w^k(0)=1, \quad \partial _n w^k(0)= w_n(x_k), \quad |\nabla w^k|\le L.$$
We can extract a subsequence of the $w^k$'s which converges uniformly on compact sets to $\bar w$. 
Moreover, by Theorem \ref{lipth}, in a ball $B_c$ with $c$ universal, the convergence holds in the $C^{1,\alpha}$ norm due to the uniform $C^{1,\alpha}$ estimates. 
In conclusion, $\bar w$ solves the same equation, and
$$\bar w_n \le L = \bar w_n(0).$$
Differentiating in the $x_n$ direction we find
\begin{equation}\label{715}
\triangle \bar w _n = \frac{\nabla h}{\bar w} \cdot \nabla \bar w_n- \frac{h}{\bar w^2} \bar w_n ,
\end{equation}
and $h, \nabla h$ are evaluated at $\nabla \bar w$. 

At the origin $\bar w_n$ has a maximum and $h(\nabla \bar w) \le 0$ by \eqref{70}. The strong maximum principle implies that $\bar w_n$ is constant in the connected component of $\{\bar w>0\}$ which contains the origin. 
In particular the point $x_0:=-e_n/L$ belongs to the free boundary of $\bar w$, and since $\bar w\in C^{2,\alpha}$ near the origin (in view of Theorem \ref{lipth}), we find that $\p \{\bar w >0\}$ is $C^{2,\alpha}$ in a neighborhood of $x_0$. This means that we can touch $\bar w$ at $x_0$ by a quadratic polynomial $P$ with $P_n \ge L-\delta$, $\triangle P >0$ in a neighborhoof of $x_0$. Then we easily contradict the definition of viscosity solutions for the $w^k$'s and reach a contradiction.

\end{proof}

We define a convex function $\eta$ in a neighborhood of $D$ which is proportional to the distance to $D$. Precisely $\eta$ is such that
$$\eta=0 \quad \mbox{on $D$}, \quad \eta(y)\sim dist(y, D), \quad \eta \in C^2(D^c),$$  and
$\|D^2 \eta\| \leq C$ universal, by the $C^2$ regularity of the domain $D$.

By compactness, from Lemma \ref{l71} above we obtain the following corollary. 
\begin{cor}\label{c7}
Assume that $w$ is a viscosity solution to \eqref{general} in $B_1$ with $0 \in F(w)$. For any $\eps>0$ there exists $\rho(\eps)>0$ small such that $$\eta(\nabla w) \le \eps \quad \mbox{ in} \quad B_\rho \cap \{w>0\}.$$ 
\end{cor}

Next, we show the following.

\begin{lem}\label{l72}
Assume that $w$ is a viscosity solution to \eqref{general} in $B_2$ and $$\eta(\nabla w) \le  \eps_0, \quad \text{ in $B_2$.}$$ Then
$$\eta(\nabla w) \le w^\xi \quad \mbox{in $B_1$,}$$
with $\eps_0$ and $\xi$ sufficiently small.
\end{lem}

\begin{proof} Let $\varphi$ be a nonnegative $C^2$ function which vanishes in $B_1$ and $\varphi=1$ on $\p B_2$. 
We show that $$g(x):=\eta(\nabla w) - w^\xi - \varphi(x)$$ cannot have a positive maximum in the region $B_1 \cap \{w>0\}$. Notice that $g\le 0$ on $\p B_1$ and, by Corollary \ref{c7}, $\limsup g \le 0$ as we approach $\p  \{w>0\}$. Assume by contradiction that $g$ achieves a positive maximum $x_0$. At $x_0$, 
\be\label{eta}
w ^\xi \le \eta (\nabla w) \le \eps_0,
\end{equation} 
and $\nabla w(x_0)$ belongs to $D^c$ and is sufficiently close to $\p D$. Then $\nabla g=0$ implies
\begin{equation}\label{72}
\partial_s(\eta(\nabla w))=\xi w^{\xi-1}w_s  + \varphi_s.
\end{equation}
At $x_0$ we compute (the functions $\eta$ and $h$ and their derivatives are evaluated at $\nabla w(x_0)$)
$$\triangle \eta(\nabla w)=\eta_k  \, \triangle w_k + \eta_{kl} \,  w_{ki}w_{li}.$$
The last term is nonnegative by the convexity of $\eta$, and after replacing $\triangle w_k$ (see \eqref{715}) we obtain
$$\triangle \eta(\nabla w) \ge\frac 1 w \eta _k h_s w_{ks} -  \frac {h}{w^2} \eta_k w_k.$$
Using $\eta_k(y)y_k \ge c$ and $h\le 0$ in $D^c$ we find that the second term is nonnegative hence
 $$\triangle \eta(\nabla w) \ge\frac 1 w  h_s \partial_s(\eta(\nabla w)) + c \frac {|h|}{w^2},$$
and by \eqref{72}, 
$$ \triangle \eta(\nabla w) \ge \xi w^{\xi-2}  h_s w_s  - C w^{-1} + c \frac {|h|}{w^2}. $$
On the other hand
$$\triangle w^\xi =  \xi w^{\xi-2}\left(h +(\xi-1) |\nabla w|^2\right) \le -c \xi w^{\xi-2},$$
hence
$$\triangle g \ge w^{\xi-2} \left(c \xi + c w^{-\xi} |h| + \xi h_s w_s - C w^{1-\xi} - C w^{2-\xi} \right)>0,$$
and we reach a contradiction.

In the last inequality we used that $w$ is sufficiently small, and then either 
$|h_s w_s| <c/2$ and the claim is clear or $C\ge |h_s w_s| \ge 1/2$ which together with $h=0$ on $\p D$ and the $C^1$ smoothness of $h$ gives (see \eqref{eta}) $|h| \ge c \eta \ge c w^\xi$, and again the inequality follows provided that $\xi$ is sufficiently small. 

\end{proof}

The lemma above leads to the following integral estimate.

\begin{lem}\label{l73}
Assume that $w$ is a viscosity solution to \eqref{general} in $B_2$ and $$\eta(\nabla w)\le \eps_0 \quad \text{in $B_2$.}$$ Then
$$\int_{B_1 \cap \{w>0\}} \frac{(\eta(\nabla w))^+}{w} dx \le C.$$
\end{lem}

\begin{proof} Let $f(t)$ be a $C^{1,1}$ smoothing of $(t^+)^{1+\xi}$, i.e 
$$f(0)=f'(0)=0, \quad f''=\min\{\eps^{-1}, t^{\xi-1}\}.$$
We have $h(\nabla w) \ge -C \eps_0$ so, if $\eps_0$ is sufficiently small then
$$\triangle f(w)=f''(w)\left ( \frac{f'}{f'' w}h+|\nabla w|^2 \right) \ge c f''(w)|\nabla w|^2 .$$
 We integrate and use
$$\int_{B_1} \triangle f(w) dx = \int_{\p B_1}\partial_\nu f(w) \le C,$$
hence
$$\int_{B_1 \cap \{|\nabla w|>c\}} f''(w) dx \le C.$$
Now the result follows by letting $\eps \to 0$ and noticing that by Lemma \ref{l72}, $$(\eta(\nabla w))^+w^{-1} \le w ^{\xi-1}=\lim_{\eps \to 0}f''(w),$$
and $\eta^+=0$ when $|\nabla w|<c$ with $c$ small. 
\end{proof}

We are now ready to show the proof of Theorem \ref{T7}.

\begin{proof}

By Corollary \ref{c7} we may assume that after some initial dilation around a free boundary point we have $\eta(\nabla w) \le \eps_0$ in $B_2$. Let $f$ be a smoothing of $t^+$, i.e.
$f(0)=f'(0)=0$ and $f''\ge 0$ supported on $[\eps, 4 \eps]$. From the computations above with this choice of $f$ we find
$$C \ge  \int_{B_1} \triangle f dx = \int_{B_1} f'' |\nabla w|^2 + f' h w^{-1}  \, dx \ge \int_{B_1} \frac c \eps \chi_{w\in [\eps,2\eps]} |\nabla w|^2 dx - C,$$
where in the last inequality we have used $h w^{-1} \ge -C \eta^+ w^{-1}$ and Lemma \ref{l73}.
 
On a ball of radius $C \eps$ around a free boundary point $z$ we have due to non-degeneracy \eqref{nondeg} and the Lipschitz continuity
$$\frac 1 \eps \int_{B_{C\eps}(z) \cap \{\eps < w <2 \eps  \}} |\nabla w|^2 dx \ge c(\kappa) \eps^{n-1},$$
and the result easily follows.

\end{proof}

We conclude the section with the following lemma.

\begin{lem}\label{reduced}
Assume $w$ satisfies the hypotheses of Theorem $\ref{T7}$ and $0 \in \p^*\{w>0\}$. If $\nu$ is the unit inner normal to $F(w)$ at $0$ then
$$ (f(\nu)x\cdot \nu -r\sigma(r))^+ \le w \le (f(\nu)x \cdot \nu+r\sigma(r))^+ \quad \mbox{in $B_r$},$$
with $\sigma(r)\to 0$ as $r \to 0$. 
\end{lem}

\begin{proof}
We need to show that any blow-up sequence of rescalings $w_r(x)=r^{-1}w(rx)$ with $r \to 0$ converges to $f(\nu) (x \cdot \nu)^+$. 
Let $\bar w$ be such a blow-up limit. Assume for simplicity of notation that $\nu=e_n$ and $f(\nu)=1$. 

The non-degeneracy and Lipschitz continuity imply that the positive set $\{w>0\}$ has positive density in any ball centered at a free boundary point.
This together with our assumption that $0 \in \p^*\{w>0\}$ gives that $0 \in F(\bar w)$ and
\be\label{720}
\bar w=0 \quad \mbox{ in $x_n \le 0$.}
\ee 
On the other hand $\nabla \bar w \in \bar D$ by Lemma \ref{l71}, and then we easily obtain $$ \bar w \le x_n^+,$$
from the convexity of $D$ and \eqref{720}. Assume by contradiction that $\bar w$ does not coincide with $x_n^+$. Then, by the strong maximum principle we have that $\bar w < x_n^+$ in $x_n > 0$. In particular we can find $\eps>0$ such that
$$\bar w \le  x_n^+ - \eps \quad \mbox{on} \quad B_1 \cap \{x_n = l\},$$
with $l$ small, universal. Now we can argue as in the proof of Lemma \ref{main} below and construct a barrier by above to conclude
$$\bar w \le (x_n - c \eps)^+ \quad \mbox{near the origin.}$$
This shows that $0$ is an interior point of $\{\bar w=0\}$ which contradicts $0 \in F(\bar w)$. 

\end{proof}

\section{Harnack Inequality}

In this section we prove a Harnack type inequality for viscosity solutions to \eqref{general}, which satisfy a flatness assumption. This will be the key ingredient in the improvement of 
flatness argument leading to the $C^{1,\alpha}$ regularity of flat free boundaries, that is Theorem \ref{flat_thm}. We follow the strategy from \cite{D}.

The constants in this section depend on the dimension $n$, the $C^1$ norm of $\p D$, the constant $\delta$ in
\eqref{fnu},\eqref{star}, and the Lipschitz norm of $h$ in a neighborhood of $\p D$.
Recall that $h$ satisfies \eqref{sign_h} which is important in our analysis since we can construct comparison subsolutions $\Psi^+$ with
$$ \triangle \Psi >0, \quad \nabla \Psi \notin \bar D \quad \Longrightarrow \quad \triangle \Psi^+ > 0 \ge \frac{h(\nabla \Psi^+)}{\Psi^+},$$
and supersolutions $\Phi^+$ with
\be\label{Phi}
 \triangle \Phi <0, \quad \nabla \Phi \in D \quad \Longrightarrow \quad \triangle \Phi^+ < 0 \le \frac{h(\nabla \Phi^+)}{\Phi^+}.
\ee
We also assume for simplicity that 
\be\label{en}
e_n \in \p D.
\ee
We wish to prove the following result.

\begin{thm}[Harnack inequality]
\label{HI} Let $w$ be a viscosity solution to \eqref{general} in $B_2$, and assume \eqref{en} holds.
There exist universal constants $\bar
\ep, \eta$,  such that if $w$ satisfies at some point $x_0 \in B_2$

\be\label{osc} (x_n+ a_0)^+ \leq w(x) \leq (x_n+ b_0)^+ \quad
\text{in $B_r(x_0) \subset B_2,$}\ee
 and
$$b_0 - a_0 \leq \ep r, $$ for some $\ep \leq \bar \ep,$ then
$$ (x_n+ a_1)^+ \leq w(x) \leq (x_n+ b_1)^+ \quad \text{in
$B_{r \eta}(x_0)$},$$ with
$$a_0 \leq a_1 \leq b_1 \leq b_0, \quad
b_1 -  a_1\leq (1-c)\ep r, $$ and $0<c<1$ universal.
\end{thm}

Before giving the proof we deduce an important consequence.

If $w$ satisfies \eqref{osc}
 with, say $r=1$, then we can apply Harnack inequality
repeatedly and obtain $$\label{osc2} (x_n+ a_m)^+ \leq w(x)
\leq (x_n+ b_m)^+ \quad \text{in $B_{\eta^{m}}(x_0)$}, $$with
$$b_m-a_m \leq (1-c)^m\ep$$ for all $m$'s such that $$(1-c)^m \eta^{-m}\ep \leq \bar
\ep.$$ This implies that for all such $m$'s, the oscillation of
the function
$$\tilde w_\ep(x) = \dfrac{w(x) -x_n }{\ep}  \quad \text{in $B_2^+(w) \cup F(w)$}$$ in
$B_{\rho}(x_0), \rho=\eta^{m}$ is less than $(1-c)^m= \eta^{\gamma m} =
\rho^\gamma$. Thus, the following corollary holds.

\begin{cor} \label{corollary}Let $w$ be as in Theorem $\ref{HI}$  satisfying \eqref{osc} for $r=1$. Then  in $B_1(x_0)$, $\tilde
w_\ep$ has a H\"older modulus of continuity at $x_0$, outside
the ball of radius $\ep/\bar \ep,$ i.e for all $x \in B_1(x_0)$, with $|x-x_0| \geq \ep/\bar\ep$
$$|\tilde w_\ep(x) - \tilde w_\ep (x_0)| \leq C |x-x_0|^\gamma.
$$
\end{cor}

The proof of the Harnack inequality relies on the following lemma.

\begin{lem}\label{main} Let $w$ be a viscosity solution to \eqref{general} in $B_1$ which satisfies \begin{equation*}  (x_n + 2\eps)^+ \geq w(x) \geq x_n^+, \quad \text{in $B_1$}.\end{equation*} There exist
universal constants $\bar \ep, \eta>0$ such that if at $\bar x=\dfrac{1}{5}e_n$ \be\label{u-p>ep2}
w(\bar x) \geq (\bar x_n + \ep)^+, \quad \eps \leq \bar \eps \ee then \be w(x) \geq
(x_n+c\eps)^+, \quad \text{in $\overline{B}_{\eta},$}\ee for some
$0<c<1$ universal. Analogously, if $$(x_n - 2\eps)^+ \leq w(x) \leq x_n^+, \quad \text{in $B_1$}$$ and $$ w(\bar x) \leq (\bar x_n - \eps)^+,$$ then $$ w(x) \leq (x_n - c \ep)^+, \quad
\text{in $\overline{B}_{\eta}.$}$$
\end{lem}

\begin{proof} We prove the first statement. The second one follows from a similar argument.

First set 
$$\tilde w:= \frac{w-x_n}{\eps} \quad  \quad \mbox{defined only in $B_2^+(w) \cup F(w)$,}$$ and 
$$\mathcal C_l:= B'_{3/4} \times \{\frac l 2 < x_n < \frac 1 2\} \subset B_1^+(w),$$
with $l$ small, universal, to be made precise later. Using that $h(e_n)=0$ and $w$ is bounded below in $\mathcal C_l$, we have
$$|\Delta \tilde w| = \frac 1 \eps|\Delta w| = \frac{1}{\eps w} |h(e_n + \eps \nabla \tilde w)|\leq C(l) |\nabla \tilde w| \quad \text{in $\mathcal C_l \cap \{|\nabla \tilde w| \le c \eps^{-1}\}$}.$$ 
This means that a sufficiently large dilation of $\tilde w$ satisfies the hypotheses of Lemma \ref{delta1} and we conclude that $|\nabla \tilde w| \le C(l)$ in the interior of $\mathcal C_l$. 
Since $$|\triangle \tilde w| \le C(l) |\nabla \tilde w| \quad\mbox{ and} \quad \tilde w \ge 0, \quad \tilde w (\bar x) \geq 1,$$ we can apply Harnack inequality and obtain
$$\tilde w \geq c(l) \quad \text{in $T_l:=B'_{1/2} \times \{x_n=l\}$},$$
that is
\be \label{top}
w \geq x_n + \eps c(l) \quad \text{on $T_l$.}
\ee

Now, let $\omega$ be the unit normal to $\Gamma$ at $e_n$ pointing towards $\R^n \setminus \bar D$, which in view of \eqref{star} satisfies $\omega_n \geq \delta.$ Set
$$Q(x):= -|x'- \frac{\omega'}{\omega_n}x_n|^2 + Ax_n^2 + x_n,$$
with $A > (n-1) + {\delta}^{-2}$ universal and define ($c=c(l)$,)
$$\Psi_t:= x_n+\eps c(Q+t), \quad t \in \R.$$
Then for $t= \underline t <0$ depending on $\delta$, $$\Psi_{\underline t} < x_n \leq w,$$ on the region $\mathcal C_{\eps}:= \bar B'_{1/2} \times \{-2\eps \leq x_n \leq l\}.$ Let $\bar t$ be the largest $t$ such that 
$$\Psi_{\bar t} \leq w \quad \text{on $\mathcal C_{\eps}$},$$ and let $\tilde x \in \mathcal C_{\eps}$ such that 
$$\Psi_{\bar t}(\tilde x)= w(\tilde x).$$ 
We show that $\bar t \geq \frac{1}{8}$. Indeed if $\bar t < \frac{1}{8}$, then for $\eps, l$ small universal, we can guarantee that 
$$\Psi_{\bar t}<0 \leq w \quad \text{on $ B'_{1/2} \times \{x_n=-2\eps\}$}, \quad \Psi_{\bar t} < x_n + \eps c \leq w, \quad \text{on $T_l$},$$
and
$$\Psi_{\bar t} < x_n \leq w, \quad \text{on $\{|x'|=1/2\} \times \{-2\eps \leq x_n \leq l\}.$}$$
We conclude that $\tilde x \in \mathcal C_{\eps}^{+}(\Psi_{\bar t}) \cup F(\Psi_{\bar t}).$
On the other hand, we argue that $\Psi_{\bar t}$ is a strict subsolution to the interior equation, and $w$ satisfies the free boundary condition, hence no touching can occur in $\mathcal C_{\eps}^{+}(\Psi_{\bar t}) \cup F(\Psi_{\bar t}),$ as long as $\nabla \Psi_{\bar t} \not \in D.$ This leads to a contradiction.

To show our claim for $\Psi_{\bar t}^+$  we check that for $\eps$ small, 
$$\Delta \Psi_{\bar t}= \eps c \Delta Q >0,$$ which follows by our choice of $A.$
We are left to prove that $\nabla \Psi_{\bar t} \not \in D$. Since 
$$\nabla \Psi_{\bar t} = e_n + \eps c \nabla Q,$$and $\omega$ is perpendicular to $\Gamma$ at $e_n,$
it is enough to show that
$$\omega \cdot \nabla Q >0.$$
A quick computation gives that for $\eps$ small,
$$\omega \cdot \nabla Q = 2A \omega_n x_n +  \omega_n >0 \quad \text{in $\mathcal C_\eps$}.$$

Thus,
$$w \geq x_n + \eps c (Q + \frac 1 8), \quad \text{on $\mathcal C_\eps$,}$$
and for $\eta$ small universal,
$$Q \geq -\frac{1}{16}, \quad \text{on $B_\eta.$}$$
This concludes our proof.

\end{proof}

We can now prove our Theorem \ref{HI}.

\smallskip

{\it Proof of Theorem \ref{HI}.} Without loss of generality, we can assume that $x_0=0, r=1$. First notice, that for $\ep$ small, if $a_0 <-1/5$ then $B_{1/10}(0)$ belongs to the zero phase of $w$, and the conclusion is trivial. Thus we only need to  distinguish two cases.

If $a_0 > 1/5$, then $B_{1/5} \subset \{w>0\}$ and 
$$0 \leq v:= \frac{w - (x_n+a_0)}{\eps} \leq 1$$
satisfies (see proof of Lemma \ref{main})
$$|\Delta v| \leq C |\nabla v| \quad \text{in $B_{1/5}$}.$$
Therefore, the claim is deduced from the standard Harnack inequality for $v.$

If $|a_0| <1/5$, we set
$$v(x):= w(x-a_0 e_n), \quad x \in B_{4/5}.$$ Then, $v$ satisfies the assumptions of Lemma \ref{main}, and the desired conclusion follows.

\qed

\section{Improvement of Flatness}

In this section we prove our main Improvement of Flatness Proposition, from which Theorem \ref{flat_thm} follows by standard 
arguments. 
The universal constants in this section depend on the dimension $n$, the $C^1$ norm of $\p D$, the constant $\delta$ in
\eqref{fnu},\eqref{star}, and the $C^1$ norm of $h$ in a neighborhood of $\p D$.

\begin{prop}Let $w$ be a viscosity solution to \eqref{general} in $B_1$.  There exist $\eps_0, r>0$ universal, such that if  $w$ is $\eps$-flat, i.e.
\be\label{flat1}
(f(e_n) x_n-\eps)^+ \leq w(x) \leq (f(e_n)x_n+\eps)^+, \quad \text{in $B_1$}, \quad \eps \leq \eps_0\ee
with $0 \in F(w)$, then  \be\label{flat2}
(f(\nu) x \cdot \nu- \frac \eps 2 r)^+ \leq w(x) \leq (f(\nu) x\cdot \nu + \frac \eps 2 r)^+ \quad \text{in $B_r,$}
\ee with $|\nu|=1$, and $ |f(\nu) \nu-f(e_n)e_n| \le C \eps,$ for $C>0$ universal. \end{prop}
\begin{proof} Without loss of generality, we assume that $f(e_n) =1.$ 

Let $r$ be fixed small (to be made precise later.) Assume by contradiction that there exist a sequence $\eps_k \to 0$ and a sequence of domains $D_k$ (and corresponding $f_k$), functions $h_k$ (satisfying the same assumptions as $f, h$ with the same bounds) and solutions $w_k$ satisfying \eqref{flat1} but not the conclusion \eqref{flat2}. Since $h_k, D_k, f_k$ have a uniformly bounded $C^1$ norm, and $\nabla h_k, \nabla f_k$ have a uniformly bounded modulus of continuity, up to extracting a subsequence,
$$h_k \to h^*, \quad D_k \to D^*, \quad f_k \to f^*$$ 
uniformly on compacts, with $h^*$ defined only in a neighborhood of $\p D^*$. The limits are also $C^1$ with 
$$\nabla h_k \to \nabla h^*, \quad \nabla f_k \to \nabla f^*$$ uniformly on compacts.
\smallskip

{\it Step 1.} Let
$$\tilde w_k :=\frac{w_k-x_n}{\eps_k} \quad \text{in $\Omega_k:= B_1^+(w_k) \cup F(u_k)$.}$$
Then, by \eqref{flat1}
\be\label{-11}-1 \leq \tilde w_k \leq 1, \quad \text{in $\Omega_k$,}\ee
and moreover $F(w_k)$ converges to $B_1 \cap \{x_n=0\}$ in the Hausdorff distance.

By Corollary \ref{corollary},  and Ascoli-Arzela, it follows that as $\eps_k \to 0$, the graphs of the $\tilde w_k$'s over $B_{1/2} \cap \Omega_k$ converge (up to a subsequence) in the Hausdorff distance to the graph of a H\"older continuous function $\tilde w$ on $B_{1/2}^+$. 

\smallskip

{\it Step 2.} We wish to show that $\tilde w$ is a viscosity solution to the linearized problem  \be\begin{cases}\label{linearized_imp}
\Delta \tilde w + v \cdot \dfrac{{\nabla \tilde w }}{x_n}=0, \quad \text{in $B_{1/2}^+,$}\\
\tilde w_\omega =0,\quad \text{on $B_{1/2} \cap \{x_n=0\},$}
\end{cases}\ee
where
$$v:= -\nabla h^*(e_n)= |\nabla h^*(e_n)| \omega,$$
and $\omega$ is the outer unit normal to $D^*$ at $e_n$. For the precise definition of viscosity solution to \eqref{linearized_imp} we refer to Section 7, where the problem above is analyzed and the necessary properties which will be used later on in this proof, are established.

Since $\tilde w_k$ satisfies
$$ \Delta \tilde w_k = \frac{1}{\eps_k} \, \frac{h_k(e_n + \eps_k \nabla \tilde w_k)-h_k(e_n)}{x_n + \eps_k \tilde w_k} \quad  \quad \text{in $\Omega_k$}$$ and $\nabla h_k \to \nabla h^*$,
Proposition 2.9 in \cite{CC} implies that $\tilde w$ satisfies the equation in the interior.

We only need to verify the free boundary condition. Following the notation in Subsection \ref{lin_sub} we set
$$s= v_n$$
and notice that in our case 
\be\label{so}
C \ge s \geq 0 \quad \quad \omega_n \ge \delta.
\ee 
In view of \eqref{-11}, the case $s \geq 1$ is trivial. Consider the case $s< 1$ and assume by contradiction that 
there exists a test function $$A |x'-\omega' \frac{x_n}{\omega_n} - \bar x'|^2+  B  + px_n^{1-s}, \quad A,B \in \R, \bar x' \in \R^{n-1}$$ with $$p < 0,$$ 
which touches $\tilde w$ by above at  $\bar x \in \{x_n=0\}.$ 
Notice that since $s \geq 0$ we can replace the test function above with
$$\phi:=A|x'-\omega' \frac{x_n}{\omega_n} - \bar x'|^2+  B  - C(A) x_n^2+ \frac p 2 x_n$$
which still touches $\tilde w$ strictly by above at $\bar x$ (in a small neighborhood) and has the property that (for $C(A)$ appropriately chosen,)
$$\Delta \phi <0.$$

Then, the convergence of the $\tilde w_k$'s to $\tilde w$ implies that there exist points in $B_{1/2} \cap \Omega_k$ with $x_k \to \bar x$ and constants $c_k \to 0$ such that 
$$\phi(x_k) + c_k = \tilde w_k(x_k)$$
and ($\mathcal N$ a small neighborhood of $x_k$)
$$\tilde w_k < \phi + c_k \quad \text{in $\mathcal N \setminus \{x_k\}.$}$$
Equivalently,
$$w_k (x_k)= (x_k)_n + \eps_k \phi_k(x_k)$$
and 
$$w_k < x_n + \eps_k(\phi + c_k) \quad \text{in $\mathcal N \setminus \{x_k\}.$}$$
Call
$$\Phi_k:=  x_n + \eps_k(\phi + c_k).$$ In order to reach a contradiction it suffices to show that $\Phi^+$ is a strict supersolution to our problem. 
Indeed (see \eqref{Phi}),
$$\Delta \Phi_{k} = \eps_k \Delta \phi <0,$$
and it remains to prove that
\be\label{outside2}\nabla \Phi_k (x) \in D_k, \quad \text{for $x$ near $x_k$.}\ee
Notice that
$$\nabla \Phi_k = e_n + \eps_k \nabla \phi,$$ and using the convergence of $D_k$ to $D^*$ it suffices to check that
$$\omega \cdot \nabla \phi <0$$ 
in a neighborhood of $\bar x$. It is easily verified that
$$\omega \cdot \nabla \psi = - 2 C(A) x_n \omega_n + \frac p 2 \omega_n,$$
and the conclusion follows since $p<0$, $\omega_n>0$.

\smallskip

{\it Step 3.} The limit function $\tilde w$ solves \eqref{linearized_imp} and $\tilde w (0)=0$ since $0 \in F(w_k)$. 
According to Theorem \ref{classical} and recalling \eqref{so} 
we find that $\tilde w$ satisfies the pointwise $C^{1,\mu}$ estimate \eqref{estimate} with universal constants. 
Thus, by the convergence of the $\tilde w_k$, we conclude that 
$$|\tilde w_k(x) - a \cdot x| \leq  C_1 r^{1+\mu}, \quad \text{in $B_{r} \cap \Omega_k$,}$$ with 
$$ |a| \leq C_0, \quad a \cdot \omega=0,$$ 
with $C_0$, $C_1$, $\mu$ universal. Hence for $r$ small enough universal, 
\be\label{tnu}x_n +\eps_k a \cdot x -\eps_k \frac r 4 \leq w_k(x) \leq x_n + \eps_k a \cdot x + \eps_k \frac r 4,  \quad \text{in $B_{r} \cap \Omega_k$.}\ee
Since 
$$a \cdot \omega=0, \quad |a| \leq C_0, \quad \Gamma^* \in C^1, \quad \Gamma_k \to \Gamma^*$$
we can write
$$e_n+ \eps_k a= \sigma_k \nu_k, \quad |\nu_k|=1,$$
with
$$|\sigma_k -f_k(\nu_k)| \leq \eps_k \frac r 4, \quad \quad f_k(\nu_k) \nu_k \in \Gamma_k$$ as long as $\eps$ is small enough. Thus, \eqref{tnu} gives
$$(f_k(\nu_k) x \cdot \nu_k-\eps_k r/2)^+ \leq w_k(x) \leq (f_k(\nu_k) x\cdot \nu_k + r \eps_k/2)^+ \quad \text{in $B_r,$}$$
and we reach a contradiction.

\end{proof}

\section{The Linearized Problem}

In this section we study the linearized problem associated to the free boundary problem \eqref{general}. This is a Neumann type problem in the upper half ball, governed by the the degenerate equation:
$$\Delta \varphi + v \cdot \frac{\nabla \varphi}{x_n}=0,$$ for some constant vector $v \in \R^n.$ We develop the viscosity theory for such problem. 

We use the following notation:
$$B_r^{+} : =B_r \cap \{x_n \geq 0\},$$
and $$B'_{r}=B_r \cap \{x_n=0\}$$  denotes a ball in $\R^{n-1}$. Points in $\R^n$ are sometimes denoted by $x=(x',x_n)$, with $x' \in \R^{n-1}.$
\subsection{The normalized linear problem.} After an affine deformation,  we reduce to the case when $v$ is parallel to $e_n$, and the operator is given by a general constant coefficients linear operator. 

Let $A=(a_{ij})_{i,j}$ be uniformly elliptic with ellipticity constants $0< \lambda \leq \Lambda$, $a_{nn}=1,$ and let $s>-1$.

\begin{defn}  We say that $\varphi$ is a viscosity subsolution in $B_1^+$ to \be\begin{cases}\label{linearized}
\mathcal L_s \varphi:=\sum_{ij}a_{ij} \varphi_{ij} + s\dfrac{\varphi_n}{x_n}=0, \quad \text{in $B_1 \cap \{x_n>0\},$}\\
\ \\
\varphi_s:=\lim_{t \to 0} \dfrac{\varphi(x'_0, t)-\varphi(x'_0, 0)}{t^{1-s}}=0 \quad \text{on $B'_1,$}
\end{cases}\ee if it is continuous in $B_1^+$, $\mathcal L_s \varphi \geq 0$ in $B_1 \cap \{x_n>0\}$ in the viscosity sense, and $\varphi$ satisfies the boundary condition in the following sense:
\begin{enumerate}
\item if $s \geq 1$, then $\varphi$ is uniformly bounded in $B_1^+;$
\item if $s <1$, then $\varphi$ is continuous in $B_1^+$ and it cannot be touched by above at a point $x_0 \in B'_1$ by a test function $$\phi:= A|x'-y'_0|^2+  B + px_n^{1-s}, \quad A,B \in \R, y'_0 \in \R^{n-1},$$ with $$p < 0.$$
\end{enumerate}
\end{defn}

Similarly we can define the notion of viscosity supersolution and viscosity solution to \eqref{linearized}.

 The main result in this section is the following theorem. From now on, 
 $$\delta^{-1} \ge s \ge -1+ \delta,$$
 and universal constants depend on $n, \delta, \lambda, \Lambda.$

\begin{thm} \label{classical_1}Let $\varphi$ be a viscosity solution to \eqref{linearized} with $|\varphi| \leq 1$ in $B_1^+.$ Then $\varphi \in C^{1,\mu}(B_{1/2}^+),$ with a universal bound on the $C^{1,\mu}$ norm. In particular, $\varphi$ satisfies for any  $x_0 \in B'_{1/2}$, 
\be\label{estimate1}|\varphi(x) -\varphi(x_0) - a'\cdot (x'-x'_0) | \leq C|x-x_0|^{1+\mu},\quad \quad |a'| \leq C,
\ee
for $C>0$, $0<\mu <1$ universal, and a vector $a'\in \R^{n-1}$ depending on $x_0$. \end{thm}

First we need to prove a H\"older regularity result.

\begin{thm}\label{hold} Let $\varphi$ be a viscosity solution to \eqref{linearized} with $|\varphi| \leq 1$ in $B_1^+.$ Then $\varphi \in C^{\alpha}(B_{1/2}^+)$, with a universal bound on the $C^\alpha$ norm.
\end{thm}

The theorem above immediately follows from the next lemma.

\begin{lem}\label{standard} Let $\varphi$ be a viscosity solution to \eqref{linearized} with $|\varphi| \leq 1$ in $B_1^+.$ Assume that 
\be\label{sep}\varphi(\frac 1 2 e_n) >0.\ee Then, there exists a universal constant $c>0$ such that 
$$\varphi \geq -1+ c \quad \text{on $B_{1/2}^+$.}$$
\end{lem}
\begin{proof} From Harnack inequality, and assumption \eqref{sep}, we get that for $l>0$ small,
\be\label{top} \varphi +1 \geq c(l), \quad \text{on $\{|x'| \leq 3/4\} \times \{x_n=l\}$}.\ee We consider first the case when $s <1.$
Let ($c:=c(l)$)
\be\label{w} 
w:= c(-|x'|^2+A x_n^2+ \frac{1}{32} x_n^{1-s}), \quad A > \Lambda \frac{(n-1)}{\delta}.
\ee 
It is easy to verify that $w$ is a strict subsolution to the interior equation in \eqref{linearized} in $B_1 \cap \{x_n>0\}.$ Moreover, if $l$ is chosen sufficiently small (depending on $A$),
\be\label{one}w \leq -\frac 1 2 c  \quad \text{on $\{|x|'=\frac 3 4\} \times \{0 \leq x_n \leq l\}$}\ee
\be\label{two}w \leq \frac 1 2 c \quad \text{on $\{|x'| \leq 3/4\} \times \{x_n=l\}$.}\ee
Now, let 
$$w_t:= w+t, \quad t \geq -T$$
with $T$ large enough so that $w_T < \varphi +1$ in $\mathcal C:=\{|x'| \leq 3/4\} \times \{0 \leq x_n \leq l\}$. Let $\bar t$ be the largest $t$ such that $w_t \leq \varphi+1$ on $\mathcal C$ and let $\bar x$ be the first contact point. We wish to show that $\bar t \geq \frac{c}{2}.$ Indeed, if that is the case then
$$w+ \frac{c}{2} \leq \varphi +1 \quad \text{on $\mathcal C.$}$$ The desired claim then would follow since 
$$w+ \frac{c}{2} \geq \ \frac{c}{4}\quad \text{on $\{|x'| \leq 1/2\} \times \{0 \leq x_n \leq l\}$.}$$ 
We are left with the proof that $\bar t \geq \frac{c}{2}.$ Indeed if $\bar t < \frac{c}{2}$, then in view of \eqref{top}-\eqref{one}-\eqref{two}, the first contact point
for $w+\bar t$ cannot occur on $\{|x|'=\frac 3 4\} \times \{0 \leq x_n \leq l\}$ or on $\{|x'| \leq 3/4\} \times \{x_n=l\}$. On the other hand, the first contact point cannot occur neither on $\{x_n=0\}$ (because of the free boundary condition), nor in the interior of $\mathcal C$ (because $w+\bar t$ is a strict subsolution to the interior equation.) We have reached a contradiction, hence the desired claim holds.

If $s \geq 1,$ we set
$$w_\eps= c(-|x'|^2+ A x_n^2-\eps x_n^{1-s}), \quad s \neq 1;$$
$$w_\eps= c(-|x'|^2+ A x_n^2+\eps \ln x_n), \quad s = 1,$$ with $$A > \Lambda \frac{n-1}{2},$$ and $\eps>0.$ We choose $d(\eps) >0$ so that $$w_\eps \leq -\frac{c}{2} \quad \text{if $x_n \leq d(\eps)$}, \quad d(\eps) \to 0 \quad \text{as $\eps \to 0$.}$$
Then it is easy to check that for $l$ small,
$$w_\eps \leq \frac{c}{2} \quad \text{on $\{|x'| \leq 3/4\} \times \{x_n=l\}$};$$
$$w_\eps \leq -\frac{c}{2} \quad \text{on $\{|x|'=\frac 3 4\} \times \{0 \leq x_n \leq l\}$}.$$
Since $\mathcal L_s w_\eps >0$, we conclude that
$$w_\eps + \frac{c}{2}\leq \varphi +1 \quad \text{in $\{|x'| \leq 3/4\} \times \{d(\eps) \leq x_n \leq l\}.$}$$
By letting $\eps \to 0$,we obtain the desired estimate.

\end{proof}

One key ingredient in the proof of Theorem \ref{classical_1} is the next proposition, from which the subsequent corollary immediately follows . We postpone its proof till the end of the section.

\begin{prop}\label{diff}Let $\varphi, \psi$ be subsolutions (resp. supersolutions) to \eqref{linearized}.Then $\varphi+\psi$ is a subsolution (resp. supersolution) to \eqref{linearized}.
\end{prop}

\begin{cor}\label{derivative}Let $\varphi$ be a viscosity solution to \eqref{linearized} then for any unit vector $e'$ in the $x'$ direction, 
$$\frac{\varphi(x+\eps e') - \varphi(x)}{\eps}$$
is a viscosity solution to \eqref{linearized}.
\end{cor}

Combining Corollary \ref{derivative} with the H\"older regularity of viscosity solutions, we obtain by standard techniques \cite{CC} the following result.

\begin{thm}\label{higher}Let $\varphi$ be a viscosity solution to \eqref{linearized} with $|\varphi| \leq 1$ in $B_1^+.$ Then for some $\mu \in (0,1)$ universal, $\varphi \in C^{k,\mu}$ in the $x'$ direction in $B_{3/4}^+$, for all $k \geq 1$, with $C^{k,\mu}$ norm bounded by a universal constant (depending on $k$).
\end{thm}

We are now ready to provide the proof of our main theorem. 

\medskip

{\it Proof of Theorem $\ref{classical_1}$}. We rewrite the interior equation in \eqref{linearized} as
$$ \varphi_{nn} + s \frac{\varphi_n}{x_n}= g(x) + h(x),$$
with
$$g(x):=-\sum_{i,j \neq n} a_{ij} \varphi_{ij}, \quad h(x):=-\sum_{i \neq n} a_{in} \varphi_{in}.$$
By Theorem \ref{higher}, the function $g(x',x_n)$ is smooth in the $x'$-direction, and in particular, it is uniformly bounded on $0 \leq x_n  \leq 1/2.$ Similarly, by interior estimates and Theorem \ref{hold}, we conclude that for some $0<\alpha<1,$
$$|h(x', x_n)| \leq C x_n^{\alpha -1}, \quad \text{in $B_{1/2}^+$}.$$ Thus, for each fixed $x' \in B'_{1/2}$, we are led to consider the ODE,
$$u'' + s \frac{u'}{t} = f(t), \quad t \in [0,1/2],$$
with 
\be\label{fbound} |f(t)| \leq C(1+t^{\alpha -1}).\ee

The general solution is given by
$$u(t)= c_1 t^{1-s} + c_2 + \bar u(t), \quad \text{for $s \neq 1$,}$$
and
$$u(t)= c_1 \ln t + c_2 + \bar u(t), \quad \text{for $s = 1$,}$$
with $\bar u(t)$ a particular solution. It is easy to check that since $f$ satisfies \eqref{fbound}, we can choose a particular solution $\bar u$ that satisfies,
$$|\bar u| \leq C t^{1+\alpha}.$$
In conclusion
$$|\varphi(x',x_n)-c_1(x')x_n^{1-s}-c_2(x')| \le C x_n^{1+\alpha}.$$

Using the smoothness of $\varphi$ in the $x'$ direction together with the free boundary condition, we conclude that $c_1\equiv 0$, $c_2(x')=\varphi(x',0)$ and \eqref{estimate1} holds.

\qed

\medskip

In order to prove Proposition \ref{diff} we also need the following expansion lemma.

\begin{lem}[Expansion at regular points]\label{expansion}Let $s<1$ and let $\varphi \in C(B_1^+)$ be a viscosity supersolution to \eqref{linearized} in $B_1^+.$ Assume that $\varphi(x',0)$ is $C^{1,1}$ at $0$ in the $x'$-direction. If $\tilde \varphi$ is a  solution to $\mathcal L_s \tilde \varphi=0$ in $B_1 \cap \{x_n>0\}$ with $\tilde \varphi=\varphi$ on $\p B_1^+$, then, $\tilde\varphi_s(0)$ is well defined and 
$$\tilde \varphi_s(0) \leq 0.$$

\end{lem}

\begin{proof} Without loss of generality, we can assume that $\tilde \varphi(0,0)=0, \nabla_{x'}\tilde \varphi(0,0)=0$. Since $\tilde \varphi(x',0)$ is $C^{1,1}$ at 0, in a neighborhood of $0$ we have that for some large constant $C>0$,
$$-C|x'|^2 \leq \tilde\varphi(x',0) \leq C|x'|^2.$$
We define $(k \geq 0)$, $$p_k:= \sup\{p: \tilde\varphi \geq - 2C|x'|^2+Ax_n^2+p x_n^{1-s} \quad \text{in $B^+_{2^{-k}}$}\},$$
$$m_k:= \inf\{m: \tilde\varphi \leq 2C|x'|^2-Ax_n^2+m x_n^{1-s} \quad \text{in $B^+_{2^{-k}}$}\},$$ with $A>0$ chosen so that
$$\mathcal L_s (-2C|x'|^2+Ax_n^2)=0.$$
Notice that $\{p_k\}_k$ is an increasing sequence, while $\{m_k\}_k$ is decreasing. Thus,
$$\bar p=\sup p_k, \quad \bar m:= \inf m_k,$$ are well defined. 

We wish to show that 
\be\label{equal}\bar p=\bar m \in (-\infty, +\infty),\ee
from which our claims will follow immediately.

First, set
$$w= -2C |x'|^2+ A x_n^2- M x_n^{1-s}$$
with $A$ as above, and $M>0$ large so that
$$w \leq \tilde\varphi \quad \text{on $\p B_1^+$}.$$ Thus, $w \leq \tilde\varphi$ in $B_1^+$ and $\{p_k\}_k$ is bounded below. Similarly, $\{m_k\}_k$ is bounded above. 
In order to obtain \eqref{equal},we prove by induction that there exist sequences $\{\bar p_k\}, \{\bar m_k\}$ with $\bar p_k \leq p_k$ and $\bar m_k \geq m_k$ such that 
\be\label{k+1}m_k-p_k \leq \bar m_{k} - \bar p_{k}= C_0 (1-c_0)^k,\ee with $c_0>0$ universal to be specified later, and $C_0$ chosen universal so that the statement holds for $k=0$.
Towards this aim let $\mu:=\frac{\bar m_k- \bar p_k}{2}$ and assume \eqref{k+1} holds for $k\geq 1$. If \be\label{assume}\tilde\varphi(\frac r 2 e_n) \geq (\bar p_k+\mu)(\frac r 2)^{1-s}, \quad r=2^{-k},\ee then we claim that 
\be\label{conclude} p_{k+1} \geq \bar p_k + c_1 \mu.\ee
Similarly, if 
\be\label{assume2}\tilde\varphi(\frac r 2 e_n) \leq (\bar m_k-\mu)(\frac r 2)^{1-s}, \quad r=2^{-k},\ee
then,
\be\label{conclude2} m_{k+1} \leq \bar m_k - c_1 \mu.\ee
Thus assuming \eqref{k+1} holds for $k \geq 1$ with $c_0 = c_1/2$, if \eqref{assume} is satisfied, we can choose $\bar p_{k+1}= \bar p_k + c_1\mu$ and $\bar m_{k+1}=\bar m_k$, otherwise we choose $\bar p_{k+1} = \bar p_k$ and $\bar m_{k+1} = \bar m_k -c_1\mu.$

To conclude our proof, let us assume that \eqref{assume} hold and let us show that \eqref{conclude} follows.

Call $$v_k := - 2C|x'|^2+Ax_n^2+\bar p_k x_n^{1-s}$$
and
$$u_k(x):=r^{-1+s} (\tilde\varphi - v_k)(r x), \quad x \in B_1.$$
Then,
$$\mathcal L_s u_k=0 , \quad u_k \geq 0 \quad \text{in $B_1 \cap \{x_n>0\}$}$$
and
$$u_k(\frac 1 2 e_n) \geq \mu -A (\frac r 2)^{s+1} \geq \frac \mu 2,$$
where in the last inequality we used that by the induction hypothesis
$$\mu = C_0(1-c_0)^k = C_0 r^\alpha$$
for some small $\alpha,$ and $C_0, c_0$ can be chosen possibly larger and smaller respectively (recall that $s+1>\delta$.) By a standard barrier argument (see proof of Lemma \ref{standard}) we conclude that 
$$u_k \geq c_1 \mu x_n^{1-s}, \quad \text{in $B_{1/2}^+$},$$ and the desired claim follows.
\end{proof}

\begin{rem} The existence of the replacement $$\tilde \varphi \in C^{2}(B_1\cap \{x_n>0\}) \cup C(\overline{B_1^+})$$ can be achieved via Perron's method. Using the barrier functions $\pm w$ in the proof above, one can guarantee the continuity up to the boundary.
\end{rem}

\medskip

We conclude this section with the proof of Proposition \ref{diff}. First, let us introduce the following regularizations. Given a continuous function $\varphi$ in $B_1^+$, we define for $\eps>0$ the upper $\eps$-envelope in the $x'$ direction,
$$\varphi^\eps(y',y_n) = \sup_{x \in B_\rho^+ \cap\{x_n =y_n\}}\{\varphi(x',y_n)  - \frac{1}{\eps}|x'-y'|^2\}, \quad y=(y',y_n) \in B_\rho^+.$$

The proof of the following facts is standard (see \cite{CC}):

(1) $\varphi^\eps \in C(B_\rho^+)$ and $\varphi_\eps \to \varphi$ uniformly in $B_\rho^+$ as $\eps \to 0.$

(2) $\varphi^\eps$ is $C^{1,1}$ in the $x'$-direction by below in $B_\rho^+$. Thus, $\varphi^\eps$ is pointwise second order differentiable in the $x'$-direction at almost every point in $B^+_\rho.$

(3) If $\varphi$ is a viscosity subsolution to \eqref{linearized} in $B_1^+$ and $r <\rho$, then for $\eps \leq \eps_0$ ($\eps_0$ depending on $\varphi, \rho, r$) $\varphi^\eps$ is a viscosity subsolution to \eqref{linearized} in $B^+_r.$ This fact follows from the obvious remark that the maximum of solutions of \eqref{linearized}  is a viscosity subsolution.

Analogously we can define $\varphi_\eps$, the lower $\eps$-envelope of $u$ in the $x'$-direction which enjoys the corresponding properties.

We are now ready  to prove our main proposition.

\medskip

\textit{Proof of Proposition \ref{diff}.} In view of property (1) above, it is enough to show that $$v:= \varphi^\eps + \psi^\eps$$ is a subsolution to \eqref{linearized} on $B^+_1$. The case $s \geq 1$ is trivial, and the interior property is standard. We only need to check the boundary condition when $s<1.$

Assume by contradiction that there exists $A > 0$ so that $$\phi:= A|x'|^2 + px_n^{1-s},$$ touches $v$ by above say at $0,$ and $p<0.$ Then $\varphi^\eps, \psi^\eps$ are $C^{1,1}$ at zero in the $x'$-direction. This follows from the fact that   $\varphi^\eps, \psi^\eps$ are $C^{1,1}$ by below (see property $(2)$) and their sum is $C^{1,1}$ by above at the origin. According to Lemma \ref{expansion}, we can consider their replacements $\tilde \varphi^\eps, \tilde \psi^\eps$. Thus $\phi$ will touch $\tilde \varphi^\eps+ \tilde \psi^\eps$ by above at zero and $$\tilde \varphi_s^\eps(0)+ \tilde \psi_s^\eps (0) \geq 0,$$ a contradiction.

\qed

\subsection{The linear problem.}\label{lin_sub} We now discuss the general case. Let  $\omega \in \mathbb S^{n}$ and $v:= \lambda \omega$, with $\lambda \in \R.$ Denote by $$s:= v \cdot e_n$$
and assume that for $\delta>0$, \be\label{delta}\delta^{-1} \geq s \geq -1 + \delta, \quad \omega_n \geq  \delta.\ee
\begin{defn}\label{v_linear}  We say that $\varphi$ is a viscosity subsolution to \be\begin{cases}\label{linearized2}
\Delta \varphi + v \cdot  \dfrac{\nabla \varphi}{x_n}=0, \quad \text{in $B_1 \cap \{x_n>0\},$}\\
\varphi_\omega:=\lim_{t \to 0} \dfrac{\varphi(x_0+t \omega)-\varphi(x_0)}{t^{1-s}} = 0 \quad \text{on $B'_1,$}
\end{cases}\ee if it is continuous in $B^+_2$, it is a subsolution to the equation in $B_1 \cap \{x_n >0\}$ in the viscosity sense, and
\begin{enumerate}
\item if $s \geq 1$, then $\varphi$ is uniformly bounded in $B_1^+;$
\item if $s <1$, then $\varphi$ is continuous in $B_1^+$ and it cannot be touched by above at a point $x_0 \in B'_1$ by a test function $$\phi:= A|x'- \frac{\omega'}{\omega_n}x_n-y'_0|^2+  B + px_n^{1-s}, \quad A,B \in \R, y'_0 \in \R^{n-1},$$ with $$p < 0.$$\end{enumerate}
\end{defn}

We remark that, after performing the following domain variation:
\be\label{tilde}\tilde \varphi(x',x_n)= \varphi(x' +\frac{\omega' x_n}{\omega_n}, x_n)\ee
 the function $\tilde\varphi$ satisfies the equation
 \be\label{dv} \sum_{i,j \neq n} d_{ij} \tilde\varphi_{ij} + \sum_{i \neq n} b_i \tilde\varphi_{in} + \tilde\varphi_{nn}+s \frac{\tilde\varphi_n}{x_n}=0 \quad \text{in $B_c^+$},\ee
 where
 \be\label{coeff} d_{ij}=  \frac{\omega_i \omega_j}{\omega_n^2}, \quad b_i= 2\frac{\omega_i}{\omega_n}.\ee
 In particular, in view of \eqref{delta}, equation \eqref{dv} is uniformly elliptic with ellipticity constants depending only on $\delta.$ It is also easy to see that $\tilde \varphi$ satisfies the free boundary condition $\tilde \varphi_s =0$ on $B'_{c}$. Thus, the next result follows from Theorem \ref{classical_1}. Here constants depending on $n,\delta,$ are called universal.

 \begin{thm} \label{classical}Let $\varphi$ be a viscosity solution to \eqref{linearized} with $|\varphi| \leq 1$ in $B_1^+.$ Then $\varphi \in C^{1,\mu}(B_{1/2}^+),$ with a universal bound on the $C^{1,\mu}$ norm. In particular, $\varphi$ satisfies for any $x_0 \in B'_{1/2},$
\be\label{estimate}|\varphi(x) -\varphi(x_0) - a\cdot (x-x_0) | \leq C|x-x_0|^{1+\mu},  \quad\quad |a| \leq C,
\ee
with $C>0$, $0<\mu <1$ universal, and a vector $a\in \R^{n-1}$ depending on $x_0$, with
$$a \cdot \omega=0.$$
 \end{thm}

\end{document}